\documentclass[11pt]{amsart}
\usepackage[margin=1.3in]{geometry} 

\usepackage[dvipsnames]{xcolor}
\usepackage{hyperref}
\hypersetup{colorlinks=true, citecolor=Blue, linkcolor=Blue} %hyperlinks
\usepackage[capitalise]{cleveref}
\usepackage{lipsum}
\usepackage{tikz-cd}
\usepackage{graphicx}
\usepackage{amssymb,amsmath}
\usepackage{mathtools}
\usepackage{mathrsfs}
\usepackage{enumerate}
\usepackage{ragged2e}
\usepackage{enumitem,pifont}
\usepackage{xypic}
\usepackage{tikz}
\usetikzlibrary{arrows,decorations.pathmorphing,automata,backgrounds}
\usetikzlibrary{backgrounds,positioning}

\usepackage{caption}
\usepackage{subcaption}
\usepackage{float}

\usepackage{egothic}
\usepackage[T1]{fontenc}
\usepackage{yfonts}

\graphicspath{{Figures/}}

\usepackage{calc}

\usepackage{dbnsymb}

\definecolor{ao}{rgb}{0.0, 0.5, 0.0}

\definecolor{darkblue}{rgb}{0,0,0.7} % darkblue color
\definecolor{green}{RGB}{57,181,74} % green color
\definecolor{violet}{RGB}{147,39,143} % violet color

\newcommand{\darkblue}{\color{darkblue}} % darkblue command

\newcommand{\defn}[1]{{\darkblue \emph{#1}}}

\newcommand{\R}{\mathbb{R}}

\newcommand{\K}{\mathbb{K}}
\newcommand{\LL}{\mathbb{L}}

\newcommand{\sA}{\mathsf{A}}
\newcommand{\sR}{\mathsf{R}}

\newcommand{\eqdef}{\mbox{\,\raisebox{0.2ex}{\scriptsize\ensuremath{\mathrm:}}\ensuremath{=}\,}} % \eqdef
\newcommand{\hdt}{\text{hdt}}

\newcommand{\trace}{\mathrm{tr}}

\newcommand{\calA}{\mathcal{A}}

\newcommand{\calF}{\mathcal{F}}

\newcommand{\Aut}{\operatorname{Aut}}
\newcommand{\TAut}{\operatorname{TAut}}

\newcommand{\Lie}{\operatorname{Lie}}
\newcommand{\gr}{\operatorname{gr}}
\newcommand{\bch}{\operatorname{bch}}

\newcounter{dummy} \numberwithin{dummy}{section}
 \newtheorem{theorem}[dummy]{Theorem}

\newtheorem{prop}[dummy]{Proposition}
\newtheorem{lemma}[dummy]{Lemma}
\newtheorem{cor}[dummy]{Corollary}
\newtheorem*{thm*}{Theorem}
\newtheorem*{prop*}{Proposition}

\theoremstyle{definition}
\newtheorem{definition}[dummy]{Definition}

\newtheorem{remark}[dummy]{Remark}

\numberwithin{equation}{section}

\usepackage{multicol}

\newcommand{\grt}{\mathsf{GRT}}
\newcommand{\kv}{\mathsf{KV}}
\newcommand{\krv}{\mathsf{KRV}}

\newcommand{\solkv}{\mathsf{SolKV}}

\newcommand{\tder}{\mathfrak{tder}}

\newcommand{\cyc}{\mathrm{cyc}}

\newcommand{\lkrv}{\mathfrak{krv}}

\newcommand{\wf}{\mathsf{wF}}

\newcommand{\arrows}{\mathsf{A}}

\newcommand{\A}{\arrows}

\usepackage{todonotes}

\title{Kashiwara--Vergne solutions degree by degree}

\author[Z. Dancso]{Zsuzsanna Dancso}
\address{School of Mathematics and Statistics\\ The University of Sydney\\ Sydney, NSW, Australia}
\email{zsuzsanna.dancso@sydney.edu.au}

\author[I. Halacheva]{Iva Halacheva}
\address{Department of Mathematics \\ Northeastern University \\ Boston, Massachusetts, USA}
\email{i.halacheva@northeastern.edu}

\author[G. Laplante-Anfossi]{Guillaume Laplante-Anfossi}
\address{School of Mathematics and Statistics \\ The University of Melbourne \\ Melbourne, Victoria, Australia}
\email{guillaume.laplanteanfossi@unimelb.edu.au}

\author[M. Robertson]{Marcy Robertson}
\address{School of Mathematics and Statistics \\ The University of Melbourne \\ Melbourne, Victoria, Australia}
\email{marcy.robertson@unimelb.edu.au}

\thanks{M.R. and G.L.A. were supported by the Australian Research Council Future Fellowship FT210100256. I.H. was supported by NSF grant DMS-2302664. The second author is also grateful for the support and hospitality of the Sydney Mathematical Research Institute (SMRI)}

\date{\today}

\begin{document}
%\lipsum
\maketitle\bibliographystyle{amsalpha}

\begin{abstract} 
We show that solutions to the Kashiwara--Vergne problem can be extended degree by degree. This can be used to simplify the computation of a class of Drinfel'd associators, which under the Alekseev--Torossian conjecture, may comprise all associators. We also give a proof that the associated graded Lie algebra of the Kashiwara--Vergne group is isomorphic to the graded Kashiwara-Vergne Lie algebra.
\end{abstract}

\section{Introduction}
The purpose of this note is to formalise a strategy for computing solutions to the Kashiwara--Vergne (KV) equations one degree at a time.  The original Kashiwara--Vergne problem was posed in the context of convolutions on Lie groups in 1978 \cite{KashiwaraVergne78}, and has wide implications from Lie theory to harmonic analysis. The first general solution was found in 2006, by Alekseev and Meinrenken \cite{AlekseevMeinrenken06}.  Later, Alekseev--Torossian \cite{AlekseevTorossian12}, reformulated the original problem to show that a KV solution is an automorphism of the degree completed free Lie algebra on two generators, satisfying two equations (see \cref{def:SolKV}).

In a series of papers \cite{AlekseevEnriquezTorossian10,AlekseevTorossian12}, Alekseev, Enriquez, and Torossian showed that, in fact, the KV equations have deep connections to deformation quantization.
In particular, the combination of \cite[Thm.~9.6]{AlekseevTorossian12} and \cite[Thm.~4 \&~5]{AlekseevEnriquezTorossian10} demonstrates that solutions to the KV equations can be explicitly constructed from Drinfel'd associators.
More recently, KV solutions were shown to be in bijection with certain invariants -- {\em formality isomorphisms} or {\em homomorphic expansions} -- in low-dimensional topology, unravelling unexpected connections with knot theory and string topology (cf. \cite{AlekseevKawazumiKunoNaef17}, \cite{Bar-NatanDancso:WKO2}).

Specifically, KV solutions are in bijection with homomorphic expansions for a class of knotted surfaces in $\mathbb{R}^4$ \cite[Thm.~4.9]{Bar-NatanDancso:WKO2}. A less direct correspondence with homomorphic expansions of the Goldman--Turaev Lie bialgebra of curves on surfaces can be used to define higher genus Kashiwara--Vergne problems
\cite{AlekseevKawazumiKunoNaef17,AlekseevKawazumiKunoNaef20}.

This paper studies the set $\solkv^{(n)}$ of solutions which satisfy the KV equations {\em up to degree $n$}. These sets of ``limited'' KV solutions are equipped with natural truncation maps $\solkv^{(n)} \to \solkv^{(n-1)}$. The main result of this note is the following theorem:

\begin{thm*}[{\cref{thm:main}}]
The truncation maps $\solkv^{(n)} \to \solkv^{(n-1)}$ are surjections, and the set of KV solutions, $\solkv$, admits a tower decomposition  \[\cdots \rightarrow \solkv^{(n+1)}\rightarrow\solkv^{(n)}\rightarrow \solkv^{(n-1)}\rightarrow\cdots\] 
\end{thm*}

A key idea to the proof is to study KV solutions via their symmetries, governed by the Kashiwara--Vergne group $\kv$ (on the left) and its graded version $\krv$ (on the right). We define a tower decomposition \[\cdots\rightarrow\krv^{(n+1)}\rightarrow \krv^{(n)}\rightarrow \krv^{(n-1)}\rightarrow \cdots\] of $\krv$, and show that each map in the tower is a surjection (\cref{lem:reduction2,lem:krv-surj}).
This follows the strategy of \cite[Thm.~5]{Bar-Natan98}, which is set in the context of Drinfel'd associators and the graded Grothendieck--Teichm\"uller group $\grt_1$. 
In fact, the proof is easier in the KV context. After passing to Lie algebras, the defining ``semi-classical hexagon'' equation for the graded Grothendieck--Teichm\"uller Lie algebra $\mathfrak{grt}_1$ in degree $n$ depends on the degree $(n-1)$ terms of the solution. 
This complexity is circumvented by showing that the semi-classical hexagon equation can be deduced from the other defining equations. 
Similar issues do not arise in the case of the defining equations of the graded KV Lie algebra $\lkrv$.

\medskip

Another important ingredient in our proof -- which enables the reduction of the problem from groups to Lie algebras  -- is the fact that the group $\krv$ can be identified with the automorphism group of the filtered, completed linear tensor category $\calA$ of arrow diagrams \cite[Thm.~5.12]{DancsoHalachevaRobertson23}.
This is part of a wider topological perspective where the Kashiwara--Vergne group $\kv$ is identified with the automorphism group of the filtered, completed tensor category of welded foams $\widehat{\wf}$ \cite[Sec.~5.2]{DancsoHalachevaRobertson23}, and where KV solutions are identified with filtered, structure preserving {\em isomorphisms} $\widehat{\wf} \rightarrow \A$ \cite[Thm.~4.9]{DancsoHalachevaRobertson23}.  
Remarkably, the filtration considered in this paper coincides with the one induced by the Vassiliev filtration
of welded foams, see \cref{rem:Vassiliev-filtration}.

The bitorsor structure on KV solutions allows one to transfer the surjectivity result from the tower decomposition of $\krv$ to the similar tower decomposition of the Kashiwara--Vergne group $\kv$: 
\[\cdots\rightarrow\kv^{(n+1)}\rightarrow \kv^{(n)}\rightarrow \kv^{(n-1)}\rightarrow\cdots \]  
 
We show that this tower decomposition induces a descending filtration of the group $\kv$ which satisfies a commutator condition (\cref{lem:commutator}); this, in turn, implies that the associated graded of $\kv$ admits a Lie bracket.
In fact, the commutator condition holds at the level of tangential automorphisms, and provides an elementary proof of the following theorem: 

\begin{thm*}[{\cref{thm:krv-lie-of-kv}}]
The associated graded Lie algebra of $\kv$ is isomorphic to the underlying Lie algebra of $\krv$. 
\end{thm*}

This was briefly observed in \cite[below Prop.~8]{AlekseevEnriquezTorossian10}, but we include a proof here because we feel it illuminates the origins of the filtration on~$\kv$. 

\medskip

Computationally, our main result implies that degree by degree calculations of KV solutions always succeed (i.e. they never ``get stuck''). 
Moreover, via the construction of \cite{AlekseevEnriquezTorossian10}, one can study Drinfel'd associators using this tower decomposition of KV solutions.
Indeed, it is known that Drinfel'd associators
are extendable degree by degree (\cite[Prop.~5.8]{Drinfeld90}, \cite[Thm.~4]{Bar-Natan98}, \cite[Prop.~10.4.9]{Fresse1}). 
However, the KV equations are stated in significantly smaller spaces than the pentagon and hexagon equations which define Drinfel'd associators, and hence degree by degree extension is less computationally demanding on the KV side than on the associator side.
Thus, {\em symmetric} KV solutions -- those which arise from Drinfel'd associators via the Alekseev--Enriquez--Torossian construction \cite{AlekseevEnriquezTorossian10} -- could be exploited to obtain explicit Drinfel'd associators to a high degree. 
Alekseev and Torossian \cite[after Prop.~4.10]{AlekseevTorossian12} conjecture that all KV solutions arise from associators.
This conjecture was verified at least to degree 16, hence \cref{thm:main} has wide computational applicability.

\subsection*{Acknowledgements} We thank Dror Bar-Natan for asking the question that grew into this note. We are grateful to Christian Haesemeyer, Peter McNamara and Arun Ram for fruitful discussions. 

%%%%%%%%%%%%%%%%%%%%%%%%%%%%%%%%%

\section{Preliminaries: the Kashiwara-Vergne Theorem}

Throughout the paper we work over a field $\K$ of characteristic zero, and write $\Lie(x,y)$ to denote the free Lie algebra on two generators over the $\K$, and  denote by $\Lie(x,y)_n$ its degree $n$ part. 
Here, a Lie word is said to be \defn{of degree $n$} if it consists of $n$ letters, e.g. $[[y, x], y]$ is a Lie word of degree three. We write $$\LL\eqdef\Pi_{n=1}^\infty \Lie(x,y)_n = \Pi_{n=1}^\infty \LL_n$$ for the degree completion of $\Lie(x,y)$.  

The universal enveloping algebra of $\LL$, can be identified -- as a completed Hopf algebra -- with the degree-completed free associative algebra generated by $x$ and $y$, denoted $\sA$. 
In this way, elements $u\in \LL$ can be viewed as elements of $\sA$, by expanding brackets as algebra commutators. As with $\LL$, the degree of a word in $\sA$ is the number of letters in the word. 

The completed graded vector space of \defn{cyclic words} in $x$ and $y$ is the {\em linear} quotient 
\[\cyc\eqdef\sA/[\sA,\sA].\]  
Here $[\sA,\sA]$ denotes the subspace of $\sA$ spanned by elements of the form $ab-ba$ for $a, b \in \sA$.  
There is a natural trace map $\trace: \sA \to \cyc$.   

A \defn{tangential derivation} of $\LL$ is a Lie derivation $u: \LL \to \LL$ for which $u(x)=[x,u_1]$ and $u(y)=[y,u_2]$ for some $u_1, u_2 \in \LL$. The commutator of tangential derivations is always a tangential derivation. Under this operation, tangential derivations of $\LL$ form a degree-completed Lie algebra, which we denote by $\tder$. There is a natural linear (not a Lie algebra) map $\LL^{\oplus 2} \to \tder$, given by $(u_1,u_2)\mapsto u$, with kernel~$\K x~\oplus~\K y$. As such, we write tangential derivations as pairs of Lie words, writing $u=(u_1,u_2)$. Homogeneous degree $n$ elements of $\tder$ are of the form $u=(u_1,u_2)$, with $u_1$ and $u_2$ both of degree $n$. The set of tangential derivations of degree $n$ is denoted $\tder_n$. 
 
The action of $\tder$ on $\LL$ extends via the product rule to a natural action $\cdot$ on $\sA$, which further descends to a natural action of $\tder$ on $\cyc$. The exponentiation of the Lie algebra $\tder$ is identified with the group of \defn{tangential automorphisms} of $\LL$, which we denote by $\TAut$. 
These are Lie automorphisms~$F:~\LL~\to~\LL$, for which $F(x)=e^{-u_1}x e^{u_1}$, and $F(y)=e^{-u_2}ye^{u_2}$, for some $(u_1, u_2) \in \tder$. 
As with $\tder$, we often write $F=(e^{u_1},e^{u_2})$. 
Note that as the linear map $\LL^{\oplus 2}\to \tder$ is not a Lie homomorphism, also for $u=(u_1,u_2)$, $e^u\neq (e^{u_1},e^{u_2})$.
The group law in $\TAut$ is defined via the $\bch$ product: for $u,v\in \tder$, by definition $e^ue^v \eqdef e^{\bch(u,v)}$, where $$\bch(u,v)=\log(e^ue^v)=u+v+\frac{1}{2}[u,v]+\cdots \in \tder$$ is the Baker-Campbell-Hausdorff Lie series. 
The action of $\tder$ on $\sA$ and $\cyc$ lifts via exponentiation to actions $\cdot$ of $\TAut$ on $\sA$ and $\cyc$.

Next, we need to recall the definitions of the divergence and Jacobian maps. Note that each element $a\in \sA$ has a unique decomposition of the form $$a=a_0+\partial_x(a)x+\partial_y(a)y$$ for some $a_0\in\K$ and $\partial_x(a),\partial_y(a)\in\sA$. In practice, $\partial_x$ picks out the words of a sum which end in $x$ and deletes their last letter $x$, and deletes all other words. 

The non-commutative \defn{divergence} map $j:\tder \rightarrow \cyc$ is the linear map defined by the following formula, for a tangential derivation $u=(u_1,u_2)$: $$j(u)\eqdef\trace( \partial_x(u_1)x+\partial_y(u_2)y) \ . $$  The divergence map is a $1$-cocycle of the Lie algebra $\tder$ with respect to the action of $\tder$ on $\cyc$: for any $u,v \in \tder$, we have $j([u,v])=u\cdot j(v)-v\cdot j(u)$ (\cite[Prop.~3.20]{AlekseevTorossian12}).  

Integrating the divergence cocycle produces the non-commutative \defn{Jacobian}: 
$$J:~\TAut~\rightarrow~\cyc,$$ 
an additive group $1$-cocycle \cite[Eq.~16]{AlekseevTorossian12}. That is, for $F,G\in\TAut$, we have $$J(FG)=J(F)+F\cdot J(G).$$
The Jacobian $J$ is uniquely determined by the two conditions  
\[J(1)=0 \quad \text{and} \quad \frac{d}{dt}\Big|_{t=0}J(e^{tu}F)=j(u)+u\cdot J(F)\] for any $F\in \TAut$ and $u\in\tder$. 

\begin{remark}
Our notation $(j,J)$ matches that of \cite{AlekseevEnriquezTorossian10,DancsoHalachevaRobertson23}, and corresponds to $(\mathsf{div},j)$ in \cite{AlekseevTorossian12} and \cite{Bar-NatanDancso:WKO2}. 
\end{remark}

\subsection{Kashiwara--Vergne solutions}\label{def:SolKV}
A \defn{Kashiwara--Vergne solution}, or KV solution for short, is a pair 
 \[ (F,r) \in \TAut \times z^2\K[[z]] \] satisfying the two equations
\begin{gather}
        F(e^xe^y)  =  e^{x+y} \tag{SolKV1}\label{eq:SolKV1}   \\
        J(F)  =  \trace(r(x+y)-r(x)-r(y)). \tag{SolKV2}\label{eq:SolKV2}
\end{gather} We denote the set of KV solutions by $\solkv$.  In any pair $(F,r)$, the tangential automorphism $F$ uniquely determines the power series $r$, and the assignment $F \mapsto r$ is called the \defn{Duflo map} (see \cite{AlekseevEnriquezTorossian10}, before Prop.~6). Hence, we often refer to a KV solution by the tangential automorphism $F$ only.

\begin{remark}
    This format of the KV equations matches the notation in \cite{AlekseevEnriquezTorossian10,DancsoHalachevaRobertson23}. However, KV solutions in \cite{AlekseevTorossian12,Bar-NatanDancso:WKO2} consist of the inverses of this version of $\solkv$. 
\end{remark}

\subsection{The Kashiwara--Vergne symmetry groups}
The set $\solkv$ is a bi-torsor under the respective left/right free and transitive actions of the Kashiwara-Vergne symmetry groups, $\kv$ and $\krv$. 

The \defn{graded Kashiwara--Vergne} group, $\krv:=\exp(\lkrv)$, as a set, consists of pairs 
 $(F,r) \in \TAut \times z^2\K[[z]]$ satisfying the equations
\begin{gather}
    F(e^{x+y}) =  e^{x+y} \tag{KRV1}; \label{eq:KRV1} \\
 J(F)  = \trace(r(x+y)-r(x)-r(y)) \tag{KRV2} \label{eq:KRV2} \ . 
\end{gather} 
Once again, $F$ uniquely determines $r$, hence $\krv$ is viewed as a subgroup of $\TAut$. 
The group $\krv$ acts freely and transitively on the right of $\solkv$ by left multiplication by the inverse \cite[Thm.~5.7]{AlekseevTorossian12}: for $G \in \krv$ and $F \in \solkv$, $F \cdot G \eqdef G^{-1}F$.

The linearisation of $\krv$ is the Lie subalgebra $\lkrv$ of $\tder$, called the \defn{graded Kashiwara-Vergne Lie algebra}, which consists of pairs $(u,r) \in \tder \times \K[[z]]$ satisfying the equations
\begin{gather}
    u(x+y) =  0 \tag{krv1} \label{eq:Lie1}\\
    j(u) = \trace(r(x+y)-r(x)-r(y)) \ . \tag{krv2} \label{eq:Lie2}
\end{gather}
This is viewed as a Lie subalgebra of $\tder$, as $u$ uniquely determines $r$ \cite[Prop.~4.5]{AlekseevTorossian12}. The Lie algebra $\lkrv$ is infinite dimensional (\cite[Thm.~4.6]{AlekseevTorossian12}), but positively graded with finite dimensional graded components, and $\exp(\lkrv)=\krv$.

\medskip

Similarly, the left symmetry group of $\solkv$ is the \defn{Kashiwara--Vergne group}, denoted $\kv$. 
As a set, $\kv$ consists of the pairs
 \[ (F,r) \in \TAut \times z^2\K[[z]] \] satisfying the equations
    \begin{gather}
        F(e^xe^y) =  e^xe^y \tag{KV1} \label{eq:KV1}\\
     J(F)  = \trace(r(\bch(x,y))-r(x)-r(y)) \tag{KV2} \label{eq:KV2} \ ,
    \end{gather} where $\bch(x,y)$ denotes the Baker--Campbell--Hausdorff series.  As with $\krv$, the group $\kv$ is a subgroup  of $\TAut$ \cite[Prop.~8]{AlekseevEnriquezTorossian10}, as $F$ uniquely determines $r$. The group $\kv$ acts freely and transitively on the left of $\solkv$ by right composition with the inverse in $\TAut$, i.e. for $G \in \kv$ and $F \in \solkv$, we have $G \cdot F \eqdef F G^{-1}$. In summary:

\begin{theorem}[{\cite[Thm.~5.7]{AlekseevTorossian12}, \cite[Prop.~8]{AlekseevEnriquezTorossian10}}]
\label{thm: action is free and transititve}
    The groups $\kv$ and $\krv$ act freely and transitively on the set $\solkv$ of KV solutions, on the left and right, respectively, and these actions commute.
\end{theorem}

%%%%%%%%%%%%%%%%%%%%%%%%%%

\section{Kashiwara--Vergne solutions degree by degree}

For each $n \geq 1$, we denote by $\LL_{\leq n}\eqdef \LL / \LL_{\geq n+1}$ the quotient of the Lie algebra $\LL$ by the ideal of elements of degree greater than $n$. Similarly, we define the degree $n$ quotients of the graded algebra $A$ and graded vector space $\cyc$, namely, $A_{\leq n}:= A/A_{\geq n+1}$, and $\cyc_{\leq n}:= \cyc/\cyc_{\geq n+1}$.

We define the degree $n$ quotient of the Lie algebra $\tder$ by setting \[\tder_{\leq n}\eqdef\tder(\LL_{\leq n}).\] 
Elements of $\tder_{\leq n}$ are Lie derivations $u : \LL_{\leq 
n} \to \LL_{\leq n}$ which act on basis vectors by \[ u(x)=[x,u_1] \quad \text{and} \quad u(y)=[y,u_2],\] where $u_1$ and $u_2$ are Lie words in $\LL_{\leq n}$. 
The commutator $uv-vu$ defines a Lie bracket on $\tder_{\leq n}$. Tangential derivations on $\LL$ descend naturally to tangential derivations of $\LL_{\leq n}$, this induces a natural surjective map of Lie algebras $\pi_n: \tder \to \tder_{\leq n}$, hence the name ``quotient''.
Note that, for $u=(u_1,u_2)\in\tder$,  $\pi_n(u)=0$ if and only if either $u=0$ in $\tder$, or the lowest degree terms of both $u_i$, $i=1,2$, are in degree $n+1$ or above. 
As before, the action of $\tder_{\leq n}$ on $\LL_{\leq n}$ extends to $A_{\leq n}$ and $\cyc_{\leq n}$.

The degree $n$ quotient of the group $\TAut$ is defined to be \[\TAut_{\leq n}\eqdef\TAut(\LL_{\leq n}).\]  Elements of $\TAut_{\leq n}$ are Lie automorphisms, $F: \LL_{\leq n} \to \LL_{\leq n}$, which act on basis elements via \[F(x)=e^{-f_1} x e^{f_1} \quad \text{and}\quad F(y)=e^{-f_2}ye^{f_2}\] for some $(e^{f_1},e^{f_2}) \in \exp(\LL)^{\oplus 2}_{\leq n}$. As for $\tder$, there is a natural surjective group homomorphism $\pi_n : \TAut \to \TAut_{\leq n}$. 

The divergence map $j: \tder \to \cyc$ is homogeneous, hence it passes to the degree $n$ quotients of $\tder$ and $\cyc$. In other words, there is a linear divergence map $j: \tder_{\leq n} \to \cyc_{\leq n}$ so that the following diagram commutes:
\begin{center}
    \begin{tikzcd}
    \tder \arrow[d, two heads,"\pi_n"] \arrow[r, "j"] & \cyc \arrow[d, two heads,"\pi_n"]\\
    \tder_{\leq n} \arrow[r, "j"]  & \cyc_{\leq n} 
    \end{tikzcd}
    \end{center}
The original 1-cocycle property of $j$ implies that the same is true on the finite degree level: for any $u,v \in \tder_{\leq n}$, $j([u,v])=u\cdot j(v)-v\cdot j(u)$.
Similarly, the non-commutative Jacobian induces a map $J: \TAut_{\leq n} \to \cyc_{\leq n}$, with the 1-cocycle property $J(FG)=J(F)+F\cdot J(G),$ for any $F,G \in \TAut_{\leq n}$.

Using the fact that $\solkv$ and $\krv$ are subsets of $\TAut$, we define degree $n$ quotients of $\solkv$ and $\krv$ as follows. 

\begin{definition}  
The restrictions of the projection maps $\pi_n : \TAut \to \TAut_{\leq n}$ to $\solkv$, respectively $\krv$, define the degree $n$ quotients
    \[ \solkv_{\leq n}\eqdef\pi_n(\solkv) \quad \text{ and } \quad \krv_{\leq n}\eqdef\pi_n(\krv). \] 
\end{definition}

\subsection{Kashiwara--Vergne solutions to finite degree} 
We now define the key object of study for this paper: KV solutions \emph{up to degree~$n$}. 

\begin{definition}
\label{def: kv solutions up to n} 
We say that $F\in\TAut_{\leq n}$ satisfies the first KV equation (\ref{eq:SolKV1}) \defn{up to degree~$n$} if it satisfies (\ref{eq:SolKV1}) in $\LL_{\leq n}$, that is: 
$$F(e^xe^y) =  e^{x+y} \quad \text{ in } \LL_{\leq n}.$$
Similarly, $F$ satisfies the second KV equation \defn{up to degree~$n$} if there exists a polynomial $r\in z^2\K[[z]]/z^{n+1}$ so that $J(F)$ satisfies the equation (\ref{eq:SolKV2}) in $\cyc_{\leq n}$, that is:
$$J(F)  = \trace(r(x+y)-r(x)-r(y))\quad \text{ in } \cyc_{\leq n}.$$
We denote the set of KV solutions (of both equations) up to degree $n$ by $\solkv^{(n)}$.
\end{definition}

\begin{remark}\label{rmk:towers}
The projection maps $\pi_n: \TAut \to \TAut_{\leq n}$ descend naturally to projections $\pi_n: \TAut_{\leq n+1} \to \TAut_{\leq n}$ This leads to a {\em tower decomposition} of $\TAut$, that is, $\TAut$ is the inverse limit of the system 
\[\cdots \rightarrow \TAut_{\leq n+1}\overset{\pi_{n+1}}{\longrightarrow} \TAut_{\leq n} \overset{\pi_n}{\longrightarrow} \TAut_{\leq n-1}\rightarrow\cdots. \] 
Given a KV solution $F\in\solkv$, the degree $n$ truncation $\pi_n(F)\in \TAut_{\leq n}$ automatically satisfies the KV equations up to degree $n$. This is also true for truncating KV solutions from degree $(n+1)$ to degree $n$: given $F\in \solkv^{(n+1)}$, the image $\pi_n(F)$ is in $\solkv^{(n)}$. In other words, $\pi_n$ restricts to a map $\pi_n: \solkv^{(n+1)} \to \solkv^{(n)}$. Thus, $\solkv$ also admits a tower decomposition
\[\cdots \rightarrow \solkv^{(n+1)}\rightarrow\solkv^{(n)}\rightarrow \solkv^{(n-1)}\rightarrow\cdots .\] 
\end{remark}

\subsection{Main theorem and proof} We are now ready to state the main theorem of this paper: namely, the maps in this tower are surjective, and as a result, given any ${F}\in\TAut_{\leq n}$ which satisfies the KV equations up to degree $n$, ${F}$ is necessarily the image of some $\widetilde{F}\in\solkv$ under the projection $\pi_{n}:\TAut\rightarrow \TAut_{\leq n}$. 
The rest of this section is devoted to proving this result:

\begin{theorem}
\label{thm:main}
    For any $n \geq 1$, the natural map $\pi_n: \solkv^{(n+1)} \rightarrow \solkv^{(n)}$ is surjective. That is, any solution up to degree $n$ can be extended to a KV solution up to degree $(n+1)$.
\end{theorem}

To prove the main theorem, we use the free and transitive action of the group $\krv$. As the Kashiwara--Vergne Lie algebra $\lkrv$ and the group $\krv$ are defined by equations taking place in graded vector spaces, we can also define {\em up to degree $n$} variants of these objects, which are only required to satisfy the relevant equations up to degree $n$. 
The key idea to the proof of Theorem~\ref{thm:main} is that given the existence of {\em some} KV solution $F$, we can use some element from the group $\krv^{(n)}$ to ``move'' the first $n$ degrees of $F$ to coincide with any given up-to-degree $n$ KV solution. This reduces the problem of extending a degree $n$ KV solution to degree $(n+1)$ to instead extending an element of the group $\krv^{(n)}$ to $\krv^{(n+1)}$. We will observe that the group $\krv$ is isomorphic to the inverse limit of a tower of subgroups \[\cdots \rightarrow \krv^{(n+1)}\rightarrow \krv^{(n)}\rightarrow\krv^{(n-1)}\rightarrow\cdots\] satisfying the defining equations of $\krv$ up to degree $n$. We prove that the groups $\krv^{(n)}$ are unipotent, and hence $\krv$ is pro-unipotent. This enables us to reduce the problem to the similar but more straightforward surjectivity question about the Lie algebra $\lkrv$. To execute this plan, we start by defining the up-to-degree $n$ versions of $\lkrv$ and $\krv$.

\begin{definition}
Let $\lkrv^{(n)}$ denote the set of $u \in \tder_{\leq n}$ for which there exists some $r\in \K[[z]]/z^{n+1}$ satisfying
\[
        u(x+y)  =  0 \text{ in } \LL_{\leq n}  \quad \text{ and } \quad
        j(u)  = \trace(r(x+y)-r(x)-r(y)) \text{ in }\cyc_{\leq n}\ .
\] 
\end{definition}

It is a straightforward calculation to check that $\lkrv^{(n)}$ is closed under the bracket inherited from $\tder_{\leq n}$, and hence it is a Lie subalgebra of $\tder_{\leq n}$.

\begin{definition}\label{def:KRVn} The up-to-degree $n$ graded Kashiwara--Vergne group, $\krv^{(n)}$, is the subset of elements $F \in \TAut_{\leq n}$ for which there exists some $r\in z^2\K[[z]]/z^{n+1}$ satisfying
\[
        F(e^{x+y})  = e^{x+y}\text{ in } \LL_{\leq n} \quad \text{and} \quad 
        J(F)  =  \trace(r(x+y)-r(x)-r(y))\text{ in } \cyc_{\leq n}.
\]
\end{definition}

\begin{prop}\label{prop: torosor up to n}
The elements of $\krv^{(n)}$ form a subgroup of $\TAut_{\leq n}$. Moreover, there exists a free and transitive right action of $\krv^{(n)}$ on $\solkv^{(n)}$ via left product with the inverse.
\end{prop}

\begin{proof}
 Since $\krv^{(n)}$ is by definition a subset of $\TAut_{\leq n}$, we only need to check that it is closed under composition and inverses. With respect to the first equation of Definition~\ref{def:KRVn} this is clear: if two automorphisms both fix $e^{x+y}$, then so does the composition, and so do the inverses.
 
For the second equation, assume that $F,G\in\krv^{(n)}$ and consequently, $\pi_n(J(F))=\pi_n(\trace(r(x+y)-r(x)-r(y)))$ and $\pi_n(J(G))=\pi_n(\trace(s(x+y)-s(x)-s(y)))$. Since $J(FG)=J(F)+F\cdot J(G)$ in $\cyc_{\leq n}$, it follows that $$\pi_n\left(F\cdot J(G))=\pi_n(F\cdot(\trace(s(x+y)-s(x)-s(y)))\right)=\pi_n(J(G)),$$ due to \eqref{eq:KV1} and the fact that conjugation cancels under the trace. Closure under inverses is similar.
    
For the free and transitive action, we can use an ``up to degree $n$'' version of Theorem~\ref{thm: action is free and transititve} (see \cite[Sec.~5.1 \& Thm.~5.7]{AlekseevTorossian12}), and the \cite{AlekseevTorossian12} proof works verbatim. In brief, given ${F}\in \solkv^{(n)}$ and $G \in \krv^{(n)}$,
it is a short direct check to see that $G^{-1} F$ satisfies \eqref{eq:SolKV1} and \eqref{eq:SolKV2} up to degree $n$. On the other hand, if ${F}$ and ${G}$ are KV solutions up to degree $n$, then again by direct calculation ${F}^{-1} {G} \in \krv^{(n)}$.
\end{proof}

\begin{remark}
    One can define analogously $\kv^{(n)}$ for the Kashiwara--Vergne group $\kv$, which also forms a subgroup of $\TAut_{\leq n}$ that acts freely and transitively on the left of $\solkv^{(n)}$ via right product with the inverse.
\end{remark}

\begin{lemma}
\label{lem:exp}
The Lie algebra $\lkrv^{(n)}$ is finite dimensional, and $\exp(\lkrv^{(n)})=\krv^{(n)}$.
\end{lemma}

\begin{proof}
This is the analogue of the fact that $\exp(\lkrv)=\krv$, essentially verbatim. Any $u\in \lkrv^{(n)}$ is an element of $\tder_{\leq n}$ and therefore $\exp(u)\in\TAut_{\leq n}$. The result now follows because the defining equations of $\krv^{(n)}$ are, by design, the exponentiation of the defining equations of $\lkrv^{(n)}$.
\end{proof}

In \cite{DancsoHalachevaRobertson23}, the authors proved that the group $\krv$ is isomorphic to a group of automorphisms of a certain tensor category. Explicitly, there is an isomorphism of groups $\krv\cong \Aut_{v}(\calA)$, where $\Aut_v(\calA)$ denotes the {\em skeleton- and expansion-preserving filtered circuit algebra automorphisms} of the linear circuit algebra of \defn{arrow diagrams}, $\calA$ (\cite[Thm.~5.12]{DancsoHalachevaRobertson23}). For the purposes of this paper, it is sufficient to know that a linear circuit algebra is a linear tensor category, that is, a linear wheeled prop or a rigid symmetric tensor category freely generated by a single object, and {\em expansion preserving} refers to post-composition. See \cite{DancsoHalachevaRobertson21} for details.  

In more detail, the proof of \cite[Thm.~5.12]{DancsoHalachevaRobertson23} constructs an injective group homomorphism 
$\Theta:\Aut_v(\calA) \to \TAut$, and shows that the image of this homomorphism is $\krv$, and there is an inverse $\Theta^{-1}: \krv \to \Aut_v(\calA)$.

For $F \in \krv$, the automorphism $\Theta^{-1}(F): \calA \to \calA$ is uniquely determined by its value on the generators of the tensor algebra $\calA$ denoted by $\YGraph$ and $\upcap$.
Thus, the construction of $\Theta^{-1}$ relies on identifying parts of $\calA$ where these values live: $\Theta^{-1}(F)(\YGraph) \in \calA(\uparrow_2)$, which is a Hopf algebra on a subset of morphisms in the tensor category $\calA$, and $\Theta^{-1}(F)(\upcap) \in \calA(\upcap)$, which is a vector space on a different subset of morphisms.

The identification with $\krv$ depends on isomorphisms $\Upsilon: \calA(\uparrow_2) \to \hat{\mathcal{U}}(\cyc \rtimes (\tder\oplus \mathfrak{a}))$
and $\kappa:\calA(\upcap) \to \cyc/\cyc_1$. 
Here the Lie algebra $\mathfrak{a}$ is the two-dimensional abelian Lie algebra, and $\cyc_1$ is the degree one part of $\cyc$. For $F=(e^{f_1},e^{f_2})$, the value of $\Theta^{-1}(F)$ on the crucial generator $\YGraph\in\calA$ is assembled from $\Upsilon^{-1}((f_1,0))$ and $\Upsilon^{-1}((0,f_2))$. 
The isomorphism $\kappa$ is used to determine the value of the generator $\upcap$ under $\Theta^{-1}(F)$.

Since $F$ is an exponential of some element of $\tder$, these values are also exponentials, and in particular given as sums of the identity and higher degree terms \cite[Def.~5.4 \& Thm.~5.12]{DancsoHalachevaRobertson23}. It follows that for any homogeneous $D \in \calA$, we have $\Theta^{-1}(F)(D)=D+\{\text{higher degree terms}\}$. 

 In \cite{DancsoHalachevaRobertson23} this is done over $\mathbb{Q}$, but it remains true over any $\K$-algebra $\sR$. We denote by $\calA(\sR)$ arrow diagrams with coefficients in $\sR$, and $\krv(\sR)$ the Kashiwara-Vergne group over $\sR$.
This categorical interpretation of $\krv$ allows us to prove the following important lemma, crucial for the reduction from groups to Lie algebras.

\begin{lemma}
\label{l:krv-unipotent}
    The group $\krv^{(n)}$ is a unipotent affine algebraic group.
\end{lemma}

\begin{proof}
Let $\sR$ be any $\K$-algebra, and let $\LL(\sR)$ be the free Lie algebra over $\sR$ generated by $x$ and $y$. The group $\krv^{(n)}(\sR)$ is defined in the same way as $\krv^{(n)}=\krv^{(n)}(\K)$, but with the coefficients understood in the algebra $\sR$ instead of $\K$. Any morphism of $\K$-algebras $f:\sR \to \sR'$ induces a morphism of groups $\krv^{(n)}(\sR) \to \krv^{(n)}(\sR')$. Thus, $\krv^{(n)}$ can be regarded as a functor from $\K$-algebras to groups, and is therefore an affine group scheme \cite[Sec.~1.2]{Waterhouse79}. Similar reasoning shows that the natural truncation maps $\pi_n:\krv^{(n+1)}\rightarrow \krv^{(n)}$, obtained by restricting the natural projections $\pi_n: \TAut_{\leq n+1}\rightarrow \TAut_{\leq n}$, are morphisms of affine group schemes.

The realisation of $\krv(\sR)$ as a group of automorphisms of a linear circuit algebra \cite[Theorem 5.12]{DancsoHalachevaRobertson23} can be adapted in a straightforward way to realise $\krv^{(n)}(\sR)$ as a group $\Aut_v(\calA(\sR)_{\leq n})$ of filtered automorphisms of the linear circuit algebra of arrow diagrams over $\sR$ of degree up to $n$.  Therefore, $\krv^{(n)}$ can be regarded as an algebraic matrix group.

Both $\Upsilon$ and $\kappa$ are degree preserving (at the level of the Lie algebras), and thus they descend to isomorphisms on the degree $n$ quotients
$\Upsilon: \calA(\uparrow_2)_{\leq n} \to \hat{\mathcal{U}}(\cyc \rtimes (\tder\oplus \mathfrak{a}))_{\leq n}$
and $\kappa:\calA(\upcap_2)_{\leq n} \to (\cyc/\cyc_1)_{\leq n}$. This allows for the construction of a faithful representation of $\krv^{(n)}$ as automorphisms of arrow diagrams
$$(\Theta^{-1})^{(n)}: \krv^{(n)} \to \Aut_v(\calA_{\leq n}),$$
defined by the same formulas as $\Theta^{-1}$. Since each generator is mapped to itself plus higher degree terms (an exponential), all elements of $\krv^{(n)}$ are unipotent, which implies that $\krv^{(n)}$ is a unipotent group (see the Theorem in \cite[Sec.~8.3]{Waterhouse79}). 
\end{proof}

The following corollary is common knowledge and stated, for example, in \cite{AlekseevEnriquezTorossian10}, but we didn't find a detailed proof in the literature. 
\begin{cor}
The group $\krv$ is a pro-unipotent group.
\end{cor}
\begin{proof}
Restricting the tower maps from the decomposition of $\TAut$ (Remark~\ref{rmk:towers}), one obtains a tower decomposition of the group $\krv$: 
\[
\cdots \rightarrow\krv^{(n+1)}\rightarrow \krv^{(n)}\rightarrow \krv^{(n-1)}\rightarrow \cdots.
\] 
\cref{l:krv-unipotent} then implies that $\krv$ is a pro-unipotent group.
\end{proof}

The following two lemmas are key ingredients in the proof of the main theorem: they establish that the truncation maps in the $\krv$ tower are surjective.

\begin{lemma}
\label{lem:krv-surj}
    The natural truncation map $\pi_n: \mathfrak{krv}^{(n+1)}\rightarrow \mathfrak{krv}^{(n)}$ is a surjective Lie algebra homomorphism.
\end{lemma}

\begin{proof}
    Consider an element $u \in \mathfrak{krv}^{(n)}$, that is, an element of $\tder_{\leq n}$, such that the equations (\ref{eq:Lie1}) and (\ref{eq:Lie2}) are satisfied up to degree $n$.
    Since in degree $n$ the equations only depend on the degree $n$ terms of $u(x)$ and $u(y)$, we can extend $u(x)$ and $u(y)$ by $0$ in degree $n+1$ to obtain $\tilde{u} \in \tder_{\leq n+1}$, and $\tilde{u}$ satisfies the equations (\ref{eq:Lie1}) and (\ref{eq:Lie2}) up to degree $n+1$. This completes the proof.
\end{proof}

\begin{lemma}
\label{lem:reduction2}
    The natural truncation map $\pi_n: \krv^{(n+1)}\rightarrow \krv^{(n)}$ is a surjective group homomorphism.
\end{lemma}

\begin{proof}
    By \cref{l:krv-unipotent}, $\krv^{(n)}$ is a unipotent algebraic group. 
    Any unipotent algebraic group over a characteristic $0$ field is connected (see the Corollary in \cite[Sec.~8.5]{Waterhouse79}), and moreover its exponential mapping is an isomorphism of algebraic varieties. 
    Thus, by Lemma~\ref{lem:exp}, the surjectivity result of Lemma~\ref{lem:krv-surj} implies surjectivity at the group level, as desired.
\end{proof}
The last remaining step before we prove \cref{thm:main}, is to note that, because $\krv^{(n)}$ acts freely and transitively on $\solkv^{(n)}$ (see Proposition \ref{prop: torosor up to n}), we can use the tower decomposition of $\krv$ to deduce information about the natural tower decomposition of $\solkv$.

We are now ready to prove the main theorem: 
\begin{proof}[Proof of {\cref{thm:main}}]
    We need to prove that the natural truncation map $$\pi_n: \solkv^{(n+1)}\to \solkv^{(n)}$$ is surjective. Let $F^{(n)}\in \solkv^{(n)}$ be a KV solution up to degree $n$; we will find an $F^{(n+1)}\in \solkv^{(n+1)}$ such that $\pi_n(F^{(n+1)})=F^{(n)}$. 
    
Choose an arbitrary solution $G$ in $\solkv$, whose existence is guaranteed by \cite{AlekseevMeinrenken06} or \cite{AlekseevTorossian12}. Its degree~$n$ truncation $G^{(n)}\eqdef\pi_n(G)$ is in $\solkv_{\leq n}$ and thus in $\solkv^{(n)}$. Moreover, $G^{(n)}$ extends to a degree $(n+1)$ solution $G^{(n+1)} \in \solkv^{(n+1)}$, namely, the degree $(n+1)$ truncation of $G$. 
    
By \cref{prop: torosor up to n}, the group~$\krv^{(n)}$ acts transitively on the set $\solkv^{(n)}$. 
Choose $H^{(n)} \in \krv^{(n)}$ such that $ G^{(n)}\cdot H^{(n)} = F^{(n)}$.
Since by \cref{lem:reduction2} the truncation $\krv^{(n+1)}\rightarrow \krv^{(n)}$ is surjective, it is possible to extend $H^{(n)}$ up one degree to $H^{(n+1)} \in \krv^{(n+1)}$.
Writing $F^{(n+1)} \eqdef G^{(n+1)} \cdot H^{(n+1)}$, we get, 
\begin{align*}
    \pi_n \left(F^{(n+1)}\right)
&=\pi_n\left(G^{(n+1)}\cdot H^{(n+1)} \right)
=\pi_n\left(\left(H^{(n+1)}\right)^{-1} G^{(n+1)}\right) \\
&= \left(H^{(n)}\right)^{-1} G^{(n)}
=F^{(n)},
\end{align*}
as required.
\end{proof}

\begin{remark}
In particular, if $\K$ is set to be the field of rational numbers $\mathbb{Q}$, and $F^{(n)}\in \solkv^{(n)}$ is a rational KV solution up to degree $n$, then according to Theorem~\ref{thm:main}, $F^{(n)}$ can be rationally extended to $F^{(n+1)}=F^{(n)}+f_{n+1}\in \solkv^{(n+1)}$. 
Therefore, there exist KV solutions with rational coefficients.
\end{remark}

%%%%%%%%%%%%%%%%%%%%%%%%%%%%%%%%%%%%%%

%\newpage

\section{The Lie algebra of the Kashiwara--Vergne group}
While the Kashiwara--Vergne group, $\kv$, is not a priori the exponentiation of a Lie algebra, we recall that for any filtered group $G$, there is an associated graded Lie algebra, as follows. Given a filtration \[G=\calF_1(G)\supset \calF_2(G)\supset\cdots \supset \calF_n(G)\supset\cdots,\] there is an associated graded vector space defined as \[\gr(G)=\bigoplus_{n\geq 1} \gr_n(G) \eqdef \bigoplus_{n\geq 1}\calF_n(G)/\calF_{n+1}(G) \ .\]
Writing $(a,b)=a^{-1}b^{-1}ab$ for the group commutator of $a,b\in G$, recall that if the filtration satisfies $(\calF_m(G),\calF_n(G))\subset \calF_{m+n}(G)$, then $\gr(G)$ admits a Lie algebra structure with bracket $[\bar{u},\bar{v}]\eqdef\overline{(u,v)}$ for $\bar{u}\in \calF_{n}(G)/\calF_{n+1}(G)$, $\bar{v}\in \calF_{m}(G)/\calF_{m+1}(G)$, and $\overline{(u,v)}\in \calF_{n+m}(G)/\calF_{n+m+1}(G)$.  
We call such a Lie algebra the \defn{associated graded Lie algebra} of the group $G$.

In this brief section we construct a filtration of the Kashiwara--Vergne group $\kv$ and show that the associated graded Lie algebra is canonically isomorphic to $\lkrv$. 

%%%%%%%%%%%%%%%%%%%%%%%%%%%%%

\subsection{A filtration on tangential automorphisms}

The group $\TAut$ admits a natural descending filtration by degree. 

\begin{definition}
The \defn{degree filtration} on $\TAut$ is given by 
\begin{equation*}
\calF_n(\TAut):=\ker\left(\TAut \to \TAut_{\leq n-1}\right)
\end{equation*}
Explicitly, $\calF_n(\TAut)$ consists of tangential automorphisms $F=(e^{f_1}, e^{f_2})$ such that for $i=1,2$, the Lie series $f_i \in \LL$ begins in degree $n$ or higher.
The degree filtration is a nested sequence of normal subgroups:
$$\TAut=\calF_0(\TAut) \supseteq \calF_1(\TAut)\supseteq \calF_2(\TAut)\supseteq \cdots \supseteq \calF_n(\TAut)\supseteq \cdots$$
\end{definition}

\begin{lemma}
\label{lem:commutator}
For any $m,n\geq 1$, we have that $$(\calF_m(\TAut),\calF_n(\TAut))\subseteq \calF_{m+n}(\TAut) \ . $$
\end{lemma}

\begin{proof}
This is a tedious, but elementary, calculation. Let $F,G\in \TAut$ be given by $F=(e^{f_1},e^{f_2})$  and $G=(e^{g_1},e^{g_2})$, for $f_1,f_2 \in \LL$ starting in degree $\geq n$ and $g_1,g_2\in \LL$ starting in degree $\geq m$. We write ``hdt'' for ``higher degree terms''. 
We aim to compute the lowest degree terms of $(F^{-1} \circ G^{-1} \circ F \circ G)(x)$ and $(F^{-1} \circ G^{-1} \circ F \circ G)(y)$.
We begin by establishing some basic formulas. First, we note that
\begin{equation*}
F(x)=x+[x,f_1]+\frac{1}{2}[[x,f_1],f_1]+\hdt,
\end{equation*}
and that composition of tangential automorphisms gives
\begin{equation*}
(F\circ G)(x)= e^{-F(g_1)}e^{-f_1}xe^{f_1}e^{F(g_1)}.
\end{equation*}
Similar formulas hold for $y$.
It follows that if $F=(e^{f_1},e^{f_2})$ and $F^{-1}=(e^{f_1'},e^{f_2'})$, then 
\begin{equation*}
F^{-1}(f_1)=-f_1' \quad \text{ and }\quad F^{-1}(f_2)=-f_2'.
\end{equation*}
For a Lie word $w$, let us denote by 
$\Big(F(w)\Big)_2$ the sum over the occurrences of $x$ and $y$ in $w$, of copies of $w$ where $x$ is replaced by $[x,f_1]$ and $y$ by $[y,f_2]$.
For example,
$\big(F\left([x,y]\right)\big)_2=\left[[x,f_1],y\right]+\left[x, [y,f_2]\right]$.
With this notation, we have
\begin{equation}
\label{eq:Fl2}
F(w)=w+\Big(F(w)\Big)_2+\hdt.
\end{equation}
Note that for $F$ chosen as above, if $w$ is of degree $k$ then $\Big(F(w)\Big)_2$ is of degree $\geq (k+n)$.
Using these formulas, one calculates directly that 
\begin{equation*}
    (F^{-1}\circ G^{-1}\circ F \circ G) (x)=
    x+ \left[x, [f_1,g_1]+\left(G^{-1}(f_1)\right)_2 + \big(F(g_1)\big)_2 \right]+\hdt.
\end{equation*}
Therefore, writing $F^{-1}\circ G^{-1}\circ F \circ G=(e^{a_1},e^{a_2})$, we have that 
\begin{equation}
\label{eq:a1}
a_1=[f_1,g_1]+\left(G^{-1}(f_1)\right)_2 + \big(F(g_1)\big)_2+\hdt.
\end{equation}
It is manifest from the formula \eqref{eq:a1} and the discussion of formula \eqref{eq:Fl2} that $a_1$ begins in degree $\geq (n+m)$. The same reasoning shows that the same is true for $a_2$, completing the proof. 
\end{proof}

\begin{definition}
    The degree filtration on $\TAut$ induces \defn{degree filtrations} on the groups $\kv$ and $\krv$: for $n \geq 0$ define
    \[
    \calF_{n}(\kv)\eqdef\kv\cap \calF_n(\TAut)
    \quad \text{and} \quad
    \calF_{n}(\krv)\eqdef \krv \cap \calF_n(\TAut).
    \] 
\end{definition}

It follows from the definitions that $\calF_n(\krv)=\ker\left(\krv\rightarrow\krv^{(n-1)}\right)$, for any $n\geq 1$.
Further, Lemma~\ref{lem:reduction2} shows that the group homomorphism $\krv \to \krv^{(n)}$ is surjective, and thus that we have $\krv^{(n)} \cong \krv / \calF_{n+1}(\krv)$. 

\begin{theorem}
\label{thm:krv-lie-of-KRV}
There is an isomorphism of Lie algebras 
\[ \gr(\krv) \cong \lkrv \]
between the associated graded Lie algebra of the group $\krv$, and the Kashiwara--Vergne Lie algebra $\lkrv$.
\end{theorem}

\begin{proof} 
For any $n\geq 1$, the degree $n$ component of $\gr(\krv)$ is the vector space $\gr_n(\krv) = \calF_n(\krv)/\calF_{n+1}(\krv)$. This is canonically isomorphic to the vector space of homogeneous degree $n$ tangential automorphisms $F=(e^{f_1},e^{f_2})$, where $f_1,f_2\in \tder_n$ and where $F$ satisfies the defining equations of $\krv$.  

Note that for $f=(f_1,f_2)\in \tder_n$, we have $e^f=(e^{f_1},e^{f_2})$ up to $\calF_{n+1}(\TAut)$. Furthermore, $f$ satisfies the defining equations of $\lkrv$, and thus there is a canonical isomorphism of vector spaces $\gr_n(\krv) \cong \lkrv_n$, where $\lkrv_n$ denotes the homogeneous degree $n$ component of the Lie algebra $\lkrv$. 

\cref{lem:commutator} implies that $\gr(\krv)$ inherits a Lie algebra structure from the associated graded Lie algebra structure on $\TAut$.  
Direct comparison shows that the bracket on the associated graded Lie algebra associated to $\TAut$ agrees with the bracket on $\tder$. It follows by restriction that the bracket on $\gr(\krv)$ agrees with the bracket on $\lkrv$. 
\end{proof}

%%%%%%%%%%%%%%%%%%%%%%%%%%%%%%%%%%%

\subsection{Kashiwara--Vergne towers}

In \cite{AlekseevEnriquezTorossian10}, just after Proposition 8, the authors observe that since $\solkv$ is a bi-torsor under the free and transitive commuting actions of $\kv$ and $\krv$, every KV solution $F\in\solkv$ induces an isomorphism, \[\krv\overset{\Psi_F}{\cong} \kv.\] 
Alternatively, one can deduce this from the identification of KV solutions with isomorphisms of completed circuit algebras in \cite[Thms.~4.9 \& 5.16]{DancsoHalachevaRobertson23} and \cite[Thm.~4.9]{Bar-NatanDancso:WKO2}. The following theorem shows that the same remains true for up-to-degree $n$ KV solutions. 

\begin{theorem}
\label{thm:towers}
Any $F^{(n)} \in \solkv^{(n)}$ induces an isomorphism \[\Psi_{F^{(n)}}: \kv^{(n)} \overset{\cong}{\rightarrow} \krv^{(n)}.\] As a consequence, the vertical arrows in the following commutative diagram are all surjective.
\begin{center}
\begin{tikzcd}
\kv \arrow[d, twoheadrightarrow]         
& \solkv \arrow[loop left, distance=3em, start anchor={[yshift=-1ex]west}, start anchor={[xshift=-2.5ex]west}, end anchor={[yshift=1ex]west}, end anchor={[xshift=-2.5ex]west} ] \arrow[d, twoheadrightarrow] \arrow[loop right, distance=3em, start anchor={[yshift=1ex]east}, start anchor={[xshift=2.5ex]east}, end anchor={[yshift=-1ex]east}, end anchor={[xshift=2.5ex]east} ]        
& \krv \arrow[d,twoheadrightarrow]         \\
\vdots \arrow[d,twoheadrightarrow]      
& \vdots \arrow[d,twoheadrightarrow]         
& \vdots \arrow[d,twoheadrightarrow]       \\
\kv^{(n+1)}  \arrow[d,twoheadrightarrow] 
& \solkv^{(n+1)} \arrow[loop left, distance=3em, start anchor={[yshift=-1ex]west}, end anchor={[yshift=1ex]west}]\arrow[d,twoheadrightarrow] \arrow[loop right, distance=3em, start anchor={[yshift=1ex]east}, end anchor={[yshift=-1ex]east} ] 
& \krv^{(n+1)} \arrow[d,twoheadrightarrow] \\
\kv^{(n)} \arrow[d,twoheadrightarrow]   
& \solkv^{(n)} \arrow[loop left, distance=3em, start anchor={[yshift=-1ex]west}, start anchor={[xshift=-1ex]west}, end anchor={[yshift=1ex]west}, end anchor={[xshift=-1ex]west} ]\arrow[d,twoheadrightarrow] \arrow[loop right, distance=3em, start anchor={[yshift=1ex]east}, start anchor={[xshift=1ex]east}, end anchor={[yshift=-1ex]east}, end anchor={[xshift=1ex]east} ]  
& \krv^{(n)} \arrow[d,twoheadrightarrow]   \\
\vdots                
& \vdots                   
& \vdots                
\end{tikzcd}
\end{center}
\end{theorem}

\begin{proof}
In \cref{prop: torosor up to n}, we show that $\krv^{(n)}$ acts freely and transitively on $\solkv^{(n)}$ on the right. 
A similar argument shows that $\kv^{(n)}$ acts freely and transitively on $\solkv^{(n)}$ on the left, and it is straightforward to check that these actions commute. The first statement is a general consequence of this bi-torsor structure of $\solkv^{(n)}$.

The vertical arrows in the diagram are the maps induced by the truncation maps $\TAut_{\leq n+1}\rightarrow \TAut_{\leq n}$. \cref{lem:reduction2,lem:krv-surj} established that the maps $\krv^{(n+1)} \to \krv^{(n)}$ are surjective group homomorphisms. Given a choice of $F\in \solkv$, and $F^{(n)}\eqdef\pi_n(F)$, there is a commutative diagram with horizontal isomorphisms
    \begin{center}
    \begin{tikzcd}
    \kv^{(n+1)} \arrow[d, two heads] \arrow[r, "\Psi_{F^{(n+1)}}"] & \krv^{(n+1)} \arrow[d, two heads]\\
    \kv^{(n)} \arrow[r, "\Psi_{F^{(n)}}"]  & \krv^{(n)} 
    \end{tikzcd}
    \end{center}
from which we deduce that the group homomorphisms $\kv^{(n+1)} \to \kv^{(n)}$ are also surjective.
The surjectivity of the maps $\solkv^{(n+1)} \to \solkv^{(n)}$ is the main theorem from the previous Section (\cref{thm:main}).
\end{proof}

\begin{cor}
The group $\kv$ is a pro-unipotent group. 
\end{cor}

\begin{proof}
According to \cref{thm:towers} above, the choice of a degree $n$ KV solution $F^{(n)}\in\solkv^{(n)}$ induces an isomorphism $\Psi_{F^{(n)}}:\krv^{(n)}\rightarrow \kv^{(n)}$. Composing the inverse of this isomorphism with the faithful representations $\left(\Theta^{-1}\right)^{(n)}:\krv^{(n)}\to \Aut_v(\calA_{\leq n})$ from Lemma~\ref{l:krv-unipotent} (based on \cite[Thm.~5.12]{DancsoHalachevaRobertson23}) one obtains a faithful representation $\kv^{(n)}\to \Aut_v(\calA_{\leq n})$. The argument that $\kv^{(n)}$ is unipotent now follows exactly as in Lemma~\ref{l:krv-unipotent}. By \cref{thm:towers}, we have $\displaystyle{\kv=\varprojlim_n \kv^{(n)}}$, and thus $\kv$ is pro-unipotent.  
\end{proof}

The following Corollary expands on the comment below \cite[Prop.~8]{AlekseevEnriquezTorossian10}.

\begin{cor}
\label{thm:krv-lie-of-kv}
There is a Lie algebra isomorphism 
\[ \gr(\kv) \cong \lkrv \]
between the associated graded Lie algebra of the Kashiwara--Vergne group $\kv$, and the Kashiwara--Vergne Lie algebra $\lkrv$.
\end{cor}

\begin{proof}
    \cref{thm:towers} shows that the graded pieces of $\krv$ and $\kv$ are isomorphic. 
    Thus, we have $\gr(\kv) \cong \lkrv$ as vector spaces.
    Since the bracket on $\gr(\kv)$ is induced by the one on $\TAut$, it follows as in \cref{thm:krv-lie-of-KRV} that the isomorphism above is in fact a Lie algebra isomorphism. 
\end{proof}

\begin{remark}
\label{rem:Vassiliev-filtration}
While outside the scope of this paper, we note that the identification of $\kv$ with automorphisms of the completed tensor category of $w$-foams in \cite{DancsoHalachevaRobertson23} aligns the degree filtration of $\kv$ with the Vassiliev filtration of a class of 4-dimensional knotted objects. Namely, $\kv$ is shown to be isomorphic to the group of skeleton- and expansion-preserving automorphisms of the completed circuit algebra of {\em w-foams} $\widehat{wF}$, which locally represent ``ribbon embeddings'' of tubes in $\R^4$ with singular vertices. The Vassiliev filtration of these knotted objects coincides with the filtration of the circuit algebra by powers of the augmentation ideal. The descending filtration of $\kv$ is compatible via the isomorphism $\kv \cong \Aut_v(\widehat{wF})$ with the filtration presented in \cite[Def.~2.25]{DancsoHalachevaRobertson23}.
\end{remark}

%%%%%%%%%%%%%%%%%%%%%%%%%%%%

\bibliography{kv}

\end{document}